\documentclass[12pt,a4paper]{amsart}
\usepackage{amsmath}
\usepackage{drpack}
\usepackage[english]{babel}
\usepackage{verbatim}
\usepackage[shortlabels]{enumitem}
\usepackage{color}
\usepackage{tikz-cd}

\newtheoremstyle{mio}%
	{}{} 
	{\itshape}{} 
	{\bfseries}{.}{ } 
	{#1 #2\thmnote{~\mdseries(#3)}} 
\theoremstyle{mio}
\newtheorem{teor}{Theorem}[section]

\newtheorem{prop}[teor]{Proposition}
\newtheorem{lemma}[teor]{Lemma}

\newtheoremstyle{definition2}%
	{}{} 
	{}{} 
	{\bfseries}{.}{ } 
	{#1 #2\thmnote{\mdseries~ #3}} 
\theoremstyle{definition2}

\newtheorem{oss}[teor]{Remark}

\newcommand{\insstar}{\mathrm{Star}}
\newcommand{\valut}{\mathbf{v}}
\newcommand{\nondiv}{\mathcal{G}}

\title{Star operations on Kunz domains}
\author{Dario Spirito}
\date{\today}
\address{Dipartimento di Matematica e Fisica, Universit\`a degli Studi
``Roma Tre'', Roma, Italy}
\email{spirito@mat.uniroma3.it}

\begin{document}
\begin{abstract}
We study star operations on Kunz domains, a class of analytically irreducible, residually rational domains associated to pseudo-symmetric numerical semigroups, and we use them to refute a conjecture of Houston, Mimouni and Park. We also find an estimate for the number of star operations in a particular case, and a precise counting in a sub-case.
\end{abstract}

\maketitle

\section{Introduction}
Let $D$ be an integral domain with quotient field $K$, and let $\insfracid(D)$ be the set of \emph{fractional ideals} of $D$, i.e., the set of $D$-submodules $I$ of $K$ such that $xI\subseteq D$ for some $x\in K\setminus\{0\}$.

A \emph{star operation} on $D$ is a map $\star:\insfracid(D)\longrightarrow\insfracid(D)$, $I\mapsto I^\star$, such that, for every $I,J\in\insfracid(D)$ and every $x\in K$:
\begin{itemize}
\item $I\subseteq I^\star$;
\item if $I\subseteq J$, then $I^\star\subseteq J^\star$;
\item $(I^\star)^\star=I^\star$;
\item $x\cdot I^\star=(xI)^\star$;
\item $D=D^\star$.
\end{itemize}
A fractional ideal $I$ is \emph{$\star$-closed} if $I=I^\star$.

The easiest example of a non-trivial star operation is the \emph{$v$-operation} $v:I\mapsto(D:(D:I))$, where if $I,J\in\insfracid(D)$ we define $(I:J):=\{x\in K\mid xJ\subseteq I\}$. An ideal that is $v$-closed is said to be \emph{divisorial}; if $I$ is divisorial and $\star$ is any other star operation then $I=I^\star$. We denote by $d$ the identity, which is obviously a star operation.

Recently, the cardinality of the set $\insstar(D)$ of the star operations on $D$ has been studied, especially in the case of Noetherian \cite{houston_noeth-starfinite,starnoeth_resinfinito} and Pr\"ufer domains \cite{twostar,hmp_finite}. In particular, Houston, Mimouni and Park started studying the relationship between the cardinality of $\insstar(D)$ and the cardinality of $\insstar(T)$, where $T$ is an overring of $D$ (an \emph{overring} of $D$ is a ring comprised between $D$ and $K$) \cite{hmp-overnoeth,hmp-overprufer}: they called a domain \emph{star regular} if $|\insstar(D)|\geq|\insstar(T)|$ for every overring of $T$. While even simple domains may fail to be star regular (for example, there are domains with just one star operations having an overring with infinitely many star operations \cite[Example 1.3]{hmp-overnoeth}), they conjectured that every one-dimensional local Noetherian domain $D$ such that $1<|\insstar(D)|<\infty$ is star regular, and proved it when the residue field of $D$ is infinite \cite[Corollary 1.18]{hmp-overnoeth}.

In this context, a rich source of examples are \emph{semigroup rings}, that is, subrings of the power series ring $K[[X]]$ (where $K$ is a field, usually finite) of the form $K[[S]]:=K[[X^S]]:=\{\sum_ia_iX^i\mid a_i=0$ for all $i\notin S\}$, where $S$ is a numerical semigroup (i.e., a submonoid $S\subseteq\insN$ such that $\insN\setminus S$ is finite). Star operations can also be defined on numerical semigroups \cite{semigruppi_main}, and there is a link between star operations on $S$ and star operations on $K[[S]]$: for example, every star operation on $S$ induces a star operation on $K[[S]]$, and $|\insstar(S)|=1$ if and only if $|\insstar(K[[S]])|=1$ \cite[Theorem 5.3]{semigruppi_main}, with the latter result corresponding to the equivalence between $S$ being symmetric and $K[[S]]$ being Gorenstein \cite{bass_ubiqGor,kunz_value}. A detailed study of star operations on some numerical semigroup rings was carried out in \cite{white-tesi-sgr}.

In this paper, we study of star operations on \emph{Kunz domains}, which are, roughly speaking, a generalization of rings in the form $K[[S]]$ where $S$ is a pseudo-symmetric semigroup (see the beginning of the next section for the definitions). We show that, if $R$ is a Kunz domain whose residue field is finite and the length of $\overline{R}/R$ is at least 4 (where $\overline{R}$ is the integral closure of $R$) then $R$ is a counterexample to Houston-Mimouni-Park's conjecture; that is, $R$ satisfies $1<|\insstar(R)|<\infty$ but there is an overring $T$ of $R$ with more star operations than $R$. In Section \ref{sect:ex}, we also study more deeply one specific class of domains, linking the cardinality of $\insstar(R)$ with the set of vector subspaces of a vector space over the residue field of $R$, and calculate the cardinality of $\insstar(R)$ when the value semigroup of $R$ is $\langle 4,5,7\rangle$.

We refer to \cite{rosales_libro} for information about numerical semigroup, and to \cite{fontana_maximality} for the passage from numerical semigroup to one-dimensional local domains.

\section{Kunz domains}
Let $S$ be a numerical semigroup, and let $g:=g(S):=\sup(\insZ\setminus S)$. We say that $S$ is a \emph{pseudo-symmetric semigroup} if $g$ is even and, for every $a\inN$, $a\neq g/2$, either $a\in S$ or $g-a\in S$. If $a_1,\ldots,a_n$ are coprime integers, we denote by $\langle a_1,\ldots,a_n\rangle$ the numerical semigroup generated by $a_1,\ldots,a_n$, i.e., $\langle a_1,\ldots,a_n\rangle=\{\lambda_1a_1+\cdots+\lambda_na_n\mid\lambda_1,\ldots,\lambda_n\inN\}$.

Let $(V,M_V)$ be a discrete valuation ring with associated valuation $\valut$. Let $(R,M_R)$ be a local subring of $V$ with the following properties:
\begin{itemize}
\item $R$ and $V$ have the same quotient field;
\item the integral closure of $R$ is $V$;
\item $R$ is Noetherian;
\item the conductor ideal $(R:V)$ is nonzero;
\item the inclusion $R\hookrightarrow V$ induces an isomorphism of residue fields $R/M_R\longrightarrow V/M_V$.
\end{itemize}
Equivalently, let $R$ be an analytically irreducible, residually rational one-dimensional Noetherian local domain having integral closure $V$. Note that for every such $R$ the set $\valut(R):=\{\valut(r)\mid r\in R\}$ is a numerical semigroup. We state explicitly a property which we will be using many times.
\begin{prop}[{\protect{\cite[Corollary to Proposition 1]{matsuoka_degree}}}]
Let $R$ be as above, and let $I\subseteq J$ be $R$-submodules of the quotient field of $R$. Then,
\begin{equation*}
\ell_R(J/I)=|\valut(J)\setminus\valut(I)|,
\end{equation*}
where $\ell_R$ is the length of an $R$-module.
\end{prop}

We say that $R$ is a \emph{Kunz domain} if $\valut(R)$ is a pseudo-symmetric semigroup \cite[Proposition II.1.12]{fontana_maximality}.

From now on, we suppose that $R$ is a Kunz domain, and we set $g:=g(\valut(R))$ and $\tau:=g/2$. The hypotheses on $R$ guarantee that, if $x\in V$ is such that $\valut(x)>g$, then $x\in R$ \cite[Theorem, p.749]{kunz_value}.

\begin{lemma}\label{lemma:defT}
Let $y\in V$ be an element of valuation $g$, and let $T:=R[y]$. Then:
\begin{enumerate}[(a)]
\item $T$ contains all elements of valuation $g$;
\item $\valut(T)=\valut(R)\cup\{g\}$;
\item $\ell_R(T/R)=1$;
\item $T=R+yR$.
\end{enumerate}
\end{lemma}
\begin{proof}
Let $y'\in V$ be another element of valuation $g$. Then, $\valut(y/y')=0$, and thus $c:=y/y'$ is a unit of $V$. Hence, there is a $c'\in R$ such that the images of $c$ and $c'$ in the residue field of $V$ coincide; in particular, $c=c'+m$ for some $m\in M_V$. Hence,
\begin{equation*}
y'=cy=(c'+m)y=c'y+my.
\end{equation*}
Since $c'\in R$, we have $c'y\in R[y]$; furthermore, $\valut(my)=\valut(m)+\valut(y)>\valut(y)=g$, and thus $my\in R$. Hence, $y'\in R[y]$, and thus $R[y]$ contains all elements of valuation $g$.

The fact that $\valut(T)=\valut(R)\cup\{g\}$ is trivial; hence, $\ell_R(T/R)=|\valut(T)\setminus\valut(R)|=1$. The last point follows from the fact that $R+yR$ is an $R$-module, from $R\subsetneq R+yR\subseteq T$ and from $\ell_R(T/R)=1$.
\end{proof}

In particular, the previous proposition shows that $T$ is independent from the element $y$ chosen. From now on, we will use $T$ to denote this ring.

\medskip

We denote by $\insfracid_0(R)$ the set of $R$-fractional ideals $I$ such that $R\subseteq I\subseteq V$. If $I$ is any fractional ideal over $R$, and $\alpha\in I$ is an element of minimal valuation, then $\alpha^{-1}I\in\insfracid_0(R)$; hence, the action of any star operation is uniquely determined by its action on $\insfracid_0(R)$. Furthermore, $V^\star=V$ for all $\star\in\insstar(R)$ (since $(R:(R:V))=V$) and thus $I^\star\in\insfracid_0(R)$ for all $I\in\insfracid_0(R)$, i.e., $\star$ restricts to a map from $\insfracid_0(R)$ to itself.

To analyze star operations, we want to subdivide them according to whether they close $T$ or not. One case is very simple.
\begin{prop}\label{prop:Tneq}
If $\star\in\insstar(R)$ is such that $T\neq T^\star$, then $\star=v$.
\end{prop}
\begin{proof}
Suppose $\star\neq v$: then, there is a fractional ideal $I\in\insfracid_0(R)$ that is $\star$-closed but not divisorial. By \cite[Lemma II.1.22]{fontana_maximality}, $\valut(I)$ is not divisorial (in $\valut(R)$) and thus by \cite[Proposition I.1.16]{fontana_maximality} there is an integer $n\in\valut(I)$ such that $n+\tau\notin\valut(I)$.

Let $x\in I$ be an element of valuation $n$, and consider the ideal  $J:=x^{-1}I\cap V$: being the intersection of two $\star$-closed ideals, it is itself $\star$-closed. Since $\valut(x)>0$, every element of valuation $g$ belongs to $J$; on the other hand, by the choice of $n$, no element of valuation $\tau$ can belong to $J$.

Consider now the ideal $L:=(R:M_R)$: then, $L$ is divisorial (since $M_R$ is divisorial) and, using \cite[Proposition II.1.16(1)]{fontana_maximality},
\begin{equation*}
\valut(L)=(\valut(R)-\valut(M_R))=\valut(R)\cup\{\tau,g\}.
\end{equation*}
We claim that $T=J\cap L$: indeed, clearly $J\cap L$ contains $R$, and if $y$ has valuation $g$ then $y\in J\cap L$ by construction; thus $T=R+yR\subseteq J\cap L$. On the other hand, $\valut(J\cap L)\subseteq\valut(J)\cap\valut(L)=\valut(R)\cup\{g\}$, and thus $J\cap L\subseteq T$.

Hence, $T=J\cap L$; since $J$ and $L$ are both $\star$-closed, so is $T$. Therefore, if $T\neq T^\star$ then $\star$ must be the divisorial closure, as claimed.
\end{proof}

Suppose now that $T=T^\star$. Then, $\star$ restricts to a star operation $\star_1:=\star|_{\insfracid(T)}$; the amount of information we lose in the passage from $\star$ to $\star_1$ depends on the $R$-fractional ideals that are not ideals over $T$. We can determine them explicitly.
\begin{lemma}\label{lemma:nonT}
Let $I\in\insfracid_0(R)$, $I\neq R$. Then, the following are equivalent.
\begin{enumerate}[(i)]
\item\label{lemma:nonT:v} $\valut(I)=\valut(R)\cup\{\tau\}$;
\item\label{lemma:nonT:g} $I$ does not contain any element of valuation $g$;
\item\label{lemma:nonT:IT} $IT\neq I$;
\item\label{lemma:nonT:can} $I$ is the canonical ideal of $R$.
\end{enumerate}
Furthermore, in this case, $I^v=(R:M_R)$.
\end{lemma}
\begin{proof}
\ref{lemma:nonT:v} $\Longrightarrow$ \ref{lemma:nonT:g} is obvious.

\ref{lemma:nonT:g} $\Longrightarrow$ \ref{lemma:nonT:IT}: since $R\subseteq I$, there is an element $x$ of $I$ of valuation $0$; hence, $IT$ contains an element of valuation $g$, and thus $IT\neq I$.

\ref{lemma:nonT:IT} $\Longrightarrow$ \ref{lemma:nonT:v}: suppose there is an $x\in I$ such that $\valut(x)\notin\valut(R)\cup\{\tau\}$. Since $\valut(R)$ is pseudo-symmetric, there is an $y\in R$ such that $\valut(y)=g-\valut(x)$; hence, $I$ contains an element (explicitly, $xy$) of valuation $g$ and, by the proof of Lemma \ref{lemma:defT}, it follows that it contains every element of valuation $g$.

Fix now an element $y\in V$ of valuation $g$. Since $IT\neq I$, there are $i\in I$, $t\in T$ such that $it\notin I$. By Lemma \ref{lemma:defT}, there are $r,r'\in R$ such that $t=r+yr'$; hence, $it=i(r+yr')=ir+iyr'$. Both $ir$ and $iyr'$ are in $I$, the former since it belongs to $IR=I$ and the latter because its valuation is at least $g$. However, this contradicts $it\notin I$; therefore, $\valut(I)\subseteq\valut(R)\cup\{\tau\}$.

If $\valut(I)=\valut(R)$, then we must have $I=R$, against our hypothesis; therefore, $\valut(I)=\valut(R)\cup\{\tau\}$.

\ref{lemma:nonT:v} $\iff$ \ref{lemma:nonT:can}: by \cite[Satz 5]{jager}, $I$ is the canonical ideal of $R$ if and only if $\valut(I)$ is the canonical ideal of $\valut(R)$. The claim follows since $\valut(R)$ is pseudo-symmetric and since the canonical ideal of a numerical semigroup $S$ is $S\cup\{x\inN\mid g(S)-x\notin S\}$, which in this case is $S\cup\{\tau\}$.

For the last claim, we first note that $(R:M_R)$ is divisorial (since $M_R$ is divisorial) and that is contains $I$: indeed, if $x\in I$ has valuation $\tau$, and $m\in M_R$, then $xm\in M_R$, for otherwise $m\notin R$ and thus $R+mR$ would be an ideal properly between $R$ and $I$, against $\ell_R(I/R)=1$. Hence, $I^v$ can only be $I$ or $(R:M_R)$. However, $(R:I)\subseteq M_R$, and thus $I^v=(R:(R:I))\supseteq(R:M_R)$. Hence, $I^v=(R:M_R)$.
\end{proof}

\begin{prop}\label{prop:Psi}
The map
\begin{equation*}
\begin{aligned}
\Psi\colon\insstar(R)\setminus\{d,v\} & \longrightarrow\insstar(T)\\
\star & \longmapsto \star|_{\insfracid(T)}
\end{aligned}
\end{equation*}
is well-defined and injective.
\end{prop}
\begin{proof}
By Proposition \ref{prop:Tneq}, if $\star\neq v$ then $T=T^\star$, and thus $\star|_{\insfracid(T)}$ is a star operation on $T$; hence, $\Psi$ is well-defined. We claim that it is injective: suppose $\star_1\neq\star_2$. Then, there is an $I\in\insfracid_0(R)$ such that $I^{\star_1}\neq I^{\star_2}$. If $I$ is a $T$-module then $\Psi(\star_1)\neq\Psi(\star_2)$; suppose $I$ is not a $T$-module.

By Lemma \ref{lemma:nonT}, $I$ can only be $R$ or a canonical ideal of $R$. In the former case, $R^{\star_1}=R=R^{\star_2}$, a contradiction. In the latter case, $I^{\star_i}$ can only if $I$ or $(R:M_R)$ (since $\ell((R:M_R)/I)=1$); suppose now that $I^\star=I$ for some $\star\in\insstar(R)$. By definition of the canonical ideal, $J=(I:(I:J))$ for every ideal $J$; since $(I:L)$ is always $\star$-closed if $I$ is $\star$-closed, it follows that $\star$ must be the identity. Since $\star_1,\star_2\neq d$, we must have $I^{\star_1}=(R:M_R)=I^{\star_2}$, against the assumptions. Thus, $\Psi$ is injective.
\end{proof}

An immediate corollary of the previous proposition is that $|\insstar(R)|\leq|\insstar(T)|+2$. Our counterexample thus involves finding star operations of $T$ that do not belong to the image of $\Psi$; to do so, we restrict to the case $\ell_R(V/R)\geq 4$ or, equivalently, $|\insN\setminus\valut(R)|\geq 4$. This excludes exactly two pseudo-symmetric numerical semigroups, namely $\langle 3,4,5\rangle$ and $\langle 3,5,7\rangle$.
\begin{lemma}\label{lemma:nonGor}
Let $S$ be a pseudo-symmetric numerical semigroup, let $g:=\max(\insN\setminus S)$ and let $S':=S\cup\{g(S)\}$. If $|\insN\setminus S|\geq 4$, then there are $a,b\in(S'-M_{S'})\setminus S'$, $a\neq b$, such that $2a,2b,a+b\in S'$.
\end{lemma}
\begin{proof}
We claim that $a:=\tau$ and $b:=g-\mu$ are the two elements we are looking for.

Since $a+M_S\subseteq S$ and $a+g>g$ (and so $a+g\in M_S$) we have $a\in(S'-M_{S'})$. Furthermore, since $|\insN\setminus S|\geq 4$, we have $g>\mu$, and thus $b+m\geq g$ for all $m\in M_{S'}$.

By the previous point, $a+m,b+m\in S'\cup\{a,b\}$ for every $m\in M_{S'}$. We always have $2a\geq g$, and thus $2a\in S'$.

If $g>2\mu$, then $a>\mu$, and so $a+b\geq g$, which implies $a+b\in S'$; moreover, also $b>\mu$, and thus $2b\in S'$.

If $g<2\mu$, then $g$ must be equal to $2\mu-2$ or to $\mu-1$; the latter case is impossible since $|\insN\setminus S|\geq 4$. Hence, $b=\mu-2$ and $a=\mu-1$. Then, $2b=2\mu-4$ and $a+b=2\mu-3$; again since $|\insN\setminus S|\geq 4$, we must have $\mu>3$, and thus $2b>a+b\geq\mu$. Furthermore, in this case $S'=\{0,\mu,\ldots\}$, and so $a+b,2b\in S'$, as claimed.
\end{proof}

\begin{prop}\label{prop:stari}
Let $K$ be the residue field of $R$, and suppose that $\ell_R(V/R)\geq 4$. There are at least $|K|+1$ star operations on $T$ that do not close $(R:M_R)$.
\end{prop}
\begin{proof}
We first note that $(R:M_R)$ is a $T$-module. Indeed, let $x\in(R:M_R)$ and $t\in T$: then, $t=r+ay$, with $r\in R$ and $\valut(y)=g$, and so $xt=xr+axy$. Both $xr$ and $axy$ belong to $(R:M_R)$, the former because $(R:M_R)$ is a $R$-module and the latter since its valuation is at least $g$: hence, $xt\in(R:M_R)$. Thus, it makes sense to ask if a star operation on $T$ closes $(R:M_R)$. We also note that $T\subsetneq(R:M_R)\subseteq(T:M_T)$, and thus $(R:M_R)^{v_T}=(T:M_T)$ (where $v_T$ is the $v$-operation on $T$).

Let $S':=\valut(T)$: by Lemma \ref{lemma:nonGor}, we can find $a,b\in(S'-M_{S'})\setminus S'$ such that $2a,2b,a+b\in S'$. Choose $x,y\in(T:M_T)$ such that $\valut(x)=a$ and $\valut(y)=b$ (and, without loss of generality, suppose $y\notin(R:M_R)$): they exist since $\valut((T:M_T))=(S'-M_{S'})$ \cite[Proposition II.1.16]{fontana_maximality}.

Let $\{\alpha_1,\ldots,\alpha_q\}$ be a complete set of representatives of $R/M_R$ (or, equivalently, of $T/M_T$), and let $T_i:=T[x+\alpha_iy]$; then, by the choice of $\valut(x)$ and $\valut(y)$, we have $T_i=T+(x+\alpha_iy)T$, and in particular $T_i\subseteq(T:M_T)$. Define $\star_i$ as the star operation
\begin{equation*}
I\mapsto I^{v_T}\cap IT_i.
\end{equation*}
We claim that $\star_i$ closes $T_i$ but not $T_j$ for $j\neq i$.

Indeed, clearly $T_i^{\star_i}=T_i$. If $j\neq i$, then $T_iT_j$ contains both $x+\alpha_iy$ and $x+\alpha_jy$, and thus it contains their difference $(\alpha_i-\alpha_j)y$. Since $\alpha_i$ and $\alpha_j$ are units corresponding to different residues, it follows that $\alpha_i-\alpha_j$ is a unit of $R$, and thus of $T$; hence, $y\in T_iT_j$. By construction, $y\in(T:M_T)$: thus, $y\in T_i^{\star_j}$. On the other hand, $y\notin T_i$, and thus $T_i^{\star_j}\neq T_i$.

Thus, $\{\star_1,\ldots,\star_q\}$ are $q=|K|$ different star operations. Furthermore, none of them closes $(R:M_R)$, since
\begin{equation*}
(R:M_R)^{\star_i}=(T:M_T)\cap(R:M_R)T[x+a_iy]
\end{equation*}
contains $y$, while $y\notin(R:M_R)$.

To conclude the proof, it is enough to note that none of the $\star_i$ are the divisorial closure (since they close one of the $T_i$, none of which are divisorial), and thus we have another star operation that does not close $(R:M_R)$.
\end{proof}

We are now ready to show that $R$ is the desired counterexample.
\begin{teor}\label{teor:kunz}
Let $R$ be a Kunz domain with finite residue field, and suppose that $\ell_R(V/R)\geq 4$. Then, $1<|\insstar(R)|<\infty$, but $R$ is not star regular.
\end{teor}
\begin{proof}
Since $K$ is a finite field and $R$ is not Gorenstein, by \cite[Theorem 2.5]{houston_noeth-starfinite} $1<|\insstar(R)|<\infty$, and the same for $T$.

By Proposition \ref{prop:Psi}, we have $|\insstar(R)|\geq 2+|\Psi(\insstar(R))|$; by Proposition \ref{prop:stari}, we have $|\Psi(\insstar(R))|\leq|\insstar(T)|-|K|-1$. Hence,
\begin{equation*}
\begin{aligned}
|\insstar(R)| & \leq 2+|\insstar(T)|-|K|-1=\\
& =|\insstar(T)|-|K|+1<|\insstar(T)|
\end{aligned}
\end{equation*}
since $|K|\geq 2$. The claim is proved.
\end{proof}

\section{The case $\valut(R)=\langle n,n+1,\ldots,2n-3,2n-1\rangle$}\label{sect:ex}
In this section, we specialize to the case of Kunz domains $R$ such that $\valut(R)=\langle n,n+1,\ldots,2n-1,2n-3\rangle=\{0,n,n+1,\ldots,2n-1,2n-3,\ldots\}$, where $n\geq 4$ is an integer. It is not hard to see that this semigroup is pseudo-symmetric, with $g=2n-2$ and $\tau=n-1$.

We note that this semigroup is pseudo-symmetric also if $n=3$, for which the number of star operations has been calculated in \cite[Proposition 2.10]{starnoeth_resinfinito}: we have $|\insstar(R)|=4$.

By Lemma \ref{lemma:nonT}, the only $I\in\insfracid_0(R)$ such that $IT\neq I$ are $R$ and the canonical ideals. From now on, we denote by $\nondiv$ the set $\{I\in\insfracid_0(R)\mid IT=I\}$; we want to parametrize $\nondiv$ by subspaces of a vector space.
\begin{lemma}\label{lemma:subsp}
Let $K$ be the residue field of $R$. Then, there is an order-preserving bijection between $\nondiv$ and the set of vector subspaces of $K^{n-1}$.
\end{lemma}
\begin{proof}
Every $I\in\nondiv$ contains $T$. The quotient of $R$-modules $\pi:V\mapsto V/T$ induces a map
\begin{equation*}
\begin{aligned}
\widetilde{\pi}\colon\nondiv & \longrightarrow\mathcal{P}(V/T)\\
I & \longmapsto \pi(I),
\end{aligned}
\end{equation*}
where $\mathcal{P}(V/T)$ denotes the power set of $V/T$. It is obvious that $\widetilde{\pi}$ is injective.

The map $\pi$ induces on $V/T$ a structure of $K$-vector space of dimension $n-1$. If $I\in\nondiv$, then its image along $\widetilde{\pi}$ will be a vector subspace; conversely, if $W$ is a vector subspace of $V/T$ then $\pi^{-1}(W)$ will be an ideal in $\nondiv$. The claim is proved.
\end{proof}

For an arbitrary domain $D$ and a fractional ideal $I$ of $D$, the star operation \emph{generated} by $I$ is the map \cite[Section 5]{semigruppi_main}
\begin{equation*}
\star_I:J\mapsto(I:(I:J))\cap J^v=J^v\cap\bigcap_{\gamma\in(I:J)\setminus\{0\}}\gamma^{-1}I;
\end{equation*}
this star operation has the property that, if $I$ is $\star$-closed for some $\star\in\insstar(R)$ and $J$ is $\star_I$-closed, then $J$ is also $\star$-closed. If $\Delta\subseteq\insfracid(S)$, we define $\star_\Delta$ as the map
\begin{equation*}
\star_\Delta:J\mapsto\bigcap_{I\in\Delta}J^{\star_I}.
\end{equation*}
In the present case, we can characterize when an ideal is $\star_\Delta$-closed.

\begin{prop}\label{prop:tauIv}
Let $I,J\in\nondiv$ and let $\Delta\subseteq\nondiv$ be a set of nondivisorial ideals.
\begin{enumerate}[(a)]
\item\label{prop:tauIv:div} $I$ is divisorial if and only if $n-1\in\valut(I)$;
\item\label{prop:tauIv:Iv} $I^v=I\cup\{x\mid\valut(x)\geq n-1\}$;
\item\label{prop:tauIv:chiuso1} if $I,J$ are nondivisorial, then $I=I^{\star_J}$ if and only if $I\subseteq \gamma^{-1}J$ for some $\gamma$ of valuation $0$;
\item\label{prop:tauIv:chiuso2} if $I$ is nondivisorial, then $I$ is $\star_\Delta$-closed if and only if $I\subseteq \gamma^{-1}J$ for some $J\in\Delta$ and some $\gamma$ of valuation $0$.
\end{enumerate}
\end{prop}
\begin{proof}
\ref{prop:tauIv:div} If $I$ is divisorial, then (since $I\neq R$) we must have $(R:M_R)\subseteq I$; in particular, $n-1\in\valut(I)$.

Suppose $n-1\in\valut(I)$; since $I$ contains every element of valuation at least $n$, it contains also all elements of valuation $n-1$. Let $x$ be such that $\valut(x)=n-1$: then, $\valut(x+r)\geq n-1$ for every $r\in V$, and thus $x+I\subseteq I$. Hence, $I$ is divisorial by \cite[Proposition II.1.23]{fontana_maximality}.

\ref{prop:tauIv:Iv} Let $L:=I\cup\{x\mid\valut(x)\geq n-1\}$. If $n-1\in\valut(I)$, then $L=I$ and $I^v=L$ by the previous point. If $n-1\notin\valut(I)$, then (since $I$ contains any element of valuation at least $n$), $L$ is a fractional ideal of $R$ such that $\valut(L)=\valut(I)\cup\{n-1\}$; hence, it is divisorial and $\ell(L/I)=1$. It follows that $L=I^v$, as claimed.

\ref{prop:tauIv:chiuso1} Suppose $I\subseteq \gamma^{-1}J$, where $\valut(\gamma)=0$. Since $J$ is not divisorial, $n-1\notin\valut(J)=\valut(\gamma^{-1}J)$; hence, using the previous point, $I=I^v\cap \gamma^{-1}J$ is closed by $\star_J$. 

Conversely, suppose $I=I^{\star_J}$. Since $I$ is nondivisorial, there must be $\gamma\in(I:J)$, $\gamma\neq 0$ such that $I\subseteq\gamma^{-1}J$ and $I^v\nsubseteq\gamma^{-1}J$. If $\valut(\gamma)>0$, then $\gamma^{-1}J$ contains the elements of valuation $n-1$; it follows that $I^v\subseteq\gamma^{-1}J$ and thus that $I^v\subseteq I^{\star_J}$, against $I=I^{\star_J}$. Hence, $\valut(\gamma)=0$, as claimed.

\ref{prop:tauIv:chiuso2} If $I\subseteq \gamma^{-1}J$ for some $J\in\Delta$ and some $\gamma$ such that $\valut(\gamma)=0$, then $I^{\star_\Delta}\subseteq I^{\star_J}=I$, and thus $I$ is $\star_\Delta$-closed.

Conversely, suppose $I=I^{\star_\Delta}$. For every $J\in\Delta$, the ideal $I^{\star_J}$ is contained in $I^v=I\cup\{x\mid\valut(x)\geq n-1\}$; since $\ell(I^v/I)=1$, it follows that $I^{\star_J}$ is either $I$ or $I^v$. Since $I=I^{\star_\Delta}$, it must be $I^{\star_J}=I$ for some $J$; by the previous point, $I\subseteq \gamma^{-1}J$ for some $\gamma$, as claimed.
\end{proof}

An important consequence of the previous proposition is the following: suppose that $\Delta$ is a set of nondivisorial ideals in $\insfracid_0(R)$ such that, when $I\neq J$ are in $\Delta$, then $I\nsubseteq\gamma^{-1}J$ for all $\gamma$ having valuation 0. Then, for every subset $\Lambda\subseteq\Delta$, the set of ideals of $\Delta$ that are $\star_\Lambda$-closed is exactly $\Lambda$; in particular, each nonempty subset of $\Delta$ generates a different star operation.

We will use this observation to estimate the cardinality of $\insstar(R)$ when the residue field is finite.
\begin{prop}\label{prop:stimafinale}
Let $R$ be a Kunz domain such that $\valut(R)=\langle n,n+1,\ldots,2n-3,2n-1\rangle$, and suppose that the residue field of $R$ has cardinality $q<\infty$.Then,
\begin{equation*}
|\insstar(R)|\geq 2^{\frac{q^{n-2}-1}{q-1}}\geq 2^{q^{n-3}}.
\end{equation*}
\end{prop}
\begin{proof}
Let $L:=\{x\in V\mid\valut(x)\geq n\}$; then, $A:=V/L$ is a $K$-algebra. Let $e_1$ be an element of valuation 1, and let $e_i:=e_1^i$; then, $\{1=e_0,e_1,\ldots,e_{n-1}\}$ projects to a $K$-basis of $A$, which for simplicity we still denote by $\{e_0,\ldots,e_{n-1}\}$. The vector subspace spanned by $e_0$ is exactly the field $K$.

Since $V$ and $L$ are stable by multiplication by every element of valuation 0, asking if $\gamma I\subseteq J$ for some $I,J\in\insfracid_0(R)$ and some $\gamma$ is equivalent to asking if there is a $\overline{\gamma}\in A$ of ``valuation'' 0 such that $\overline{\gamma}\overline{I}\subseteq\overline{J}$, where $\overline{I}$ and $\overline{J}$ are the images of $I$ and $J$, respectively, in $A$. Hence, instead of working with ideals in $\insfracid_0(R)$ we can work with vector subspaces of $A$ containing $e_0$.

Furthermore, if $V$ is a vector subspace of $A$ and $\gamma$ has valuation 0, then $\gamma V$ has the same dimension of $V$; thus, if $V$ and $W$ have the same dimension, $\gamma V\subseteq W$ if and only if $\gamma V=W$. Let $\sim$ denote the equivalence relation such that $V\sim W$ if and only if $\gamma V=W$ for some $\gamma$ of valuation 0.

Let $X$ be the set of 2-dimensional subspaces of $A$ that contain $e_0$ but not $e_{n-1}$. Then, the preimage of every element of $X$ is a nondivisorial ideal.

An element of $X$ is in the form $\langle e_0,\lambda_1e_1+\cdots+\lambda_{n-1}e_{n-1}\rangle$, where at least one among $\lambda_1,\ldots,\lambda_{n-2}$ is not 0; since $\langle e_0,f\rangle=\langle e_0,\lambda f\rangle$ for all $\lambda\in K$, $\lambda\neq 0$, there are exactly $(q^{n-1}-q)/(q-1)$ such subspaces.

Let $V\in X$, say $V=\langle e_0,f\rangle$, and consider the equivalence class $\Delta$ of $V$ with respect to $\sim$. Then, $W\in\Delta$ if and only if $\gamma W=V$ for some $\gamma$; since $1\in W$, it follows that such a $\gamma$ must belong to $V$. Since $\gamma$ has valuation 0, it must be in the form $\lambda_0e_0+\lambda_1f$ with $\lambda_0\neq 0$; furthermore, if $\gamma'=\lambda\gamma$ then $\gamma^{-1}V=\gamma'^{-1}W$. Hence, the cardinality of $\Delta$ is at most $\frac{q^2-q}{q-1}=q$.

Therefore, $X$ contains elements belonging to at least
\begin{equation*}
\inv{q}\frac{q^{n-1}-q}{q-1}=\frac{q^{n-2}-1}{q-1}\geq q^{n-3}
\end{equation*}
equivalence classes; let $X'$ be a set of representatives of such classes, and let $Y$ be the preimage of $X'$ in the power set of $\insfracid_0(R)$. Then, every subset of $Y$ generates a different star operation (with the empty set corresponding to the $v$-operation); it follows that
\begin{equation*}
|\insstar(R)|\geq 2^{\frac{q^{n-2}-1}{q-1}}\geq 2^{q^{n-3}},
\end{equation*}
as claimed.
\end{proof}

For $n=4$, we can even calculate $|\insstar(R)|$.
\begin{prop}
Let $R$ be a Kunz domain such that $\valut(R)=\langle 4,5,7\rangle$, and suppose that the residue field of $R$ has cardinality $q<\infty$. Then, $|\insstar(R)|=2^{2q}+3$.
\end{prop}
\begin{proof}
Consider the same setup of the previous proof. We start by claiming that two vector subspaces $W_1,W_2$ of $A$ of dimension 3 that contain $e_0$ but not $e_3$ are equivalent under $\sim$.

Indeed, any such subspace must have a basis of the form $\{e_0,e_1+\theta_1e_3,e_2+\theta_2e_3\}$, and different pairs $(\theta_1,\theta_2)$ induce different subspaces; let $W(\theta_1,\theta_2):=\langle e_0,e_1+\theta_1e_3,e_2+\theta_2e_3\rangle$. To show that two such subspaces are equivalent, we prove that they are all equivalent to $W(0,0)$. Let $\gamma:=e_0-\theta_2e_1-\theta_1e_2$: we claim that $\gamma W(\theta_1,\theta_2)=W(0,0)$. Indeed, $\gamma e_0=\gamma\in W(0,0)$; on the other hand,
\begin{equation*}
\gamma(e_1+\theta_1e_3)=e_1+\theta_1e_3-\theta_2e_2-\theta_1e_3=e_1-\theta_2e_2\in W(0,0),
\end{equation*}
and likewise
\begin{equation*}
\gamma(e_2+\theta_2e_3)=e_2+\theta_2e_3-\theta_2e_3=e_2\in W(0,0).
\end{equation*}
Hence, $W(\theta_1,\theta_2)\sim W(0,0)$.

\medskip

Consider now the set $\Delta$ of nondivisorial ideals in $\insfracid_0(R)$. By Lemma \ref{lemma:defT} and Proposition \ref{prop:tauIv}, $\Delta$ is equal to the union of the set of the canonical ideals and the set $\nondiv$ of the $I\in\insfracid_0(R)$ such that $IT=T$. By Lemma \ref{lemma:subsp} and Proposition \ref{prop:tauIv}, the elements of the latter correspond to the subspaces of $V/T$ containing $e_0$ but not $e_3$: hence, we can write  $\nondiv=\nondiv_1\cup\nondiv_2\cup\nondiv_3$, where $\nondiv_i$ contains the ideals of $\nondiv$ corresponding to subspaces of dimension $i$.

Given $\star\in\insstar(R)$, let $\Delta(\star):=\{I\in\Delta\mid I=I^\star\}$. We claim that $\Delta(\star)$ is one of the following:
\begin{itemize}
\item $\Delta$;
\item $\Delta\setminus\{J\}$;
\item $\Lambda\cup\{T\}$ for some $\Lambda\subseteq\nondiv_2$;
\item the empty set.
\end{itemize}

By Proposition \ref{prop:Tneq}, if $T\neq T^\star$ (i.e., if $T\notin\Delta(\star)$) then $\star=v$, and $\Delta(\star)=\emptyset$.

If $\Delta(\star)$ contains a canonical ideal then $\star$ is the identity, and thus $\Delta(\star)=\Delta$.

If $I$ is $\star$-closed for some $I\in\nondiv_3$, then every element of $\nondiv_3$ must be closed, since any other $I'\in\nondiv_3$ is in the form $\gamma I$ for some $\gamma$ of valuation 0 (by the first part of the proof); furthermore, every element of $\nondiv_2$ is the intersection of the elements of $\nondiv_3$ containing it, and thus it is $\star$-closed. It follows that $\Delta(\star)=\Delta\setminus\{J\}$; in particular, there is only one such star operation.

Let $\star$ be any star operation different from the three above. Then, $\Delta(\star)$ must contain $T$ and cannot contain any canonical ideal nor any element of $\nondiv_3$. Hence, $\Delta(\star)$ must be equal to $\Lambda\cup\{T\}$ for some $\Lambda\subseteq\nondiv_2$. Moreover, $\Lambda\cup\{T\}$ is equal to $\Delta(\star)$ for some $\star$ if and only if $\Lambda$ is the (possibly empty) union of equivalence classes under $\sim$. It follows that $|\insstar(R)|=2^x+3$, where $x$ is the number of such equivalence classes.

\medskip

By the proof of Proposition \ref{prop:stimafinale}, the image of an element of $\nondiv_2$ is in the form $\langle e_0,f\rangle$, where $f=\lambda_1e_1+\lambda_2e_2+\lambda_3e_3$ with at least one between $\lambda_1$ and $\lambda_2$ nonzero. Let $V(\lambda_1,\lambda_2,\lambda_3)$ denote the subspace $\langle e_0,f\rangle$; clearly, $V(\lambda_1,\lambda_2,\lambda_3)=V(c\lambda_1,c\lambda_2,c\lambda_3)$ for every $c\in K\setminus\{0\}$. The subspaces equivalent to $V$ must have the form $(e_0+\theta f)^{-1}V$ for some $\theta\in K$, and, by using the basis $\{e_0,e_0+\theta f\}$ of $V$, we see that $(e_0+\theta f)^{-1}V(\lambda_1,\lambda_2,\lambda_3)=\langle e_0,(e_0+\theta f)^{-1}\rangle$. If $\theta=0$, then $e_0+\theta f=e_0$, and thus $(e_0+\theta f)^{-1}V(\lambda_1,\lambda_2,\lambda_3)=V(\lambda_1,\lambda_2,\lambda_3)$; suppose, from now on, that $\theta\neq 0$.

To calculate $(e_0+\theta f)^{-1}=e_0+\alpha_1e_1+\alpha_2e_2+\alpha_3e_3$, we can simply expand the product $(e_0+\theta f)(e_0+\alpha_1e_1+\alpha_2e_2+\alpha_3e_3)$, using $e_i=0$ for $i>3$, and then impose that the coefficients of $e_1$, $e_2$ and $e_3$ are zero; we obtain
\begin{equation*}
\begin{cases}
\alpha_1=-\theta\lambda_1\\
\alpha_2=-\theta(\lambda_1\alpha_1+\lambda_2)\\
\alpha_3=-\theta(\lambda_1\alpha_2+\lambda_2\alpha_1+\lambda_3).
\end{cases}
\end{equation*}
Since $\theta\neq 0$, the set $\{e_0,(e_0+\theta f)^{-1}-e_0\}$ is a basis of $(e_0+\theta f)^{-1}V(\lambda_1,\lambda_2,\lambda_3)$; hence, $(e_0+\theta f)^{-1}V(\lambda_1,\lambda_2,\lambda_3)=V(\alpha_1,\alpha_2,\alpha_3)$. We distinguish two cases.

If $\lambda_1=0$, then $\lambda_2\neq 0$, and so we can suppose $\lambda_2=1$. Then, we have
\begin{equation*}
\begin{cases}
\alpha_1=0\\
\alpha_2=-\theta\\
\alpha_3=-\theta\lambda_3.
\end{cases}
\end{equation*}
and so $(e_0+\theta f)^{-1}V(0,1,\lambda_3)=V(0,-\theta,-\theta\lambda_3)=V(0,1,\lambda_3)$ since $\theta\neq 0$. It follows that the only subspace equivalent to $V(0,1,\lambda_3)$ is $V(0,1,\lambda_3)$ itself; since we have $q$ choices for $\lambda_3$, this case gives $q$ different equivalence classes.

If $\lambda_1\neq 0$, we can suppose $\lambda_1=1$. Then, we get
\begin{equation*}
\begin{cases}
\alpha_1=-\theta\\
\alpha_2=-\theta(\alpha_1+\lambda_2)=-\theta(-\theta+\lambda_2)\\
\alpha_3=-\theta(-\theta(-\theta+\lambda_2)-\theta\lambda_2+\lambda_3).
\end{cases}
\end{equation*}
Since $\theta\neq 0$, we can divide by $-\theta$, obtaining 
\begin{equation*}
(e_0+\theta f)^{-1}V(1,\lambda_2,\lambda_3)=V(1,-\theta+\lambda_2,\theta^2-2\theta\lambda_2+\lambda_3).
\end{equation*}
Since $-\theta+\lambda_2\neq-\theta'+\lambda_2$ if $\theta\neq\theta'$, we have $(e_0+\theta f)^{-1}V(1,\lambda_2,\lambda_3)\neq(e_0+\theta' f)^{-1}V(1,\lambda_2,\lambda_3)$ for all $\theta\neq\theta'$; thus, every equivalence class is composed by $q$ subspaces. Since there are $q^2$ such subspaces, we get other $q$ equivalence classes.

Therefore, $\nondiv_2$ is partitioned into $2q$ equivalence classes, and so $|\insstar(R)|=2^{2q}+3$, as claimed.
\end{proof}

\begin{oss}
~\begin{enumerate}
\item The estimate obtained in Proposition \ref{prop:stimafinale} grows very quickly; for example, if $q$ is fixed, it follows that the double logarithm of $|\insstar(R)|$ grows (at least) linearly in $n=\ell(V/R)+1$. This should be compared with \cite[Theorem 3.21]{starnoeth_resinfinito}, where the authors analyzed a case where the growth of $|\insstar(R)|$ was linear in $\ell(V/R)$.
\item Thanks to Theorem \ref{teor:kunz}, Proposition \ref{prop:stimafinale} also gives lower bounds for the cardinality of the set of star operations of $T:=R\cup\{x\in V\mid\valut(x)=2n-2\}$. If $V=K[[X]]$ is the ring of power series, then $T$ will be equal to $K+X^nK[[X]]$. In particular, for $n=4$, we have $|\insstar(T)|\geq 2^{2q}+3$, which is an estimate pretty close to the precise cardinality of $\insstar(T)$, namely $2^{2q+1}+2^{q+1}+2$ \cite[Corollary 4.1.2]{white-tesi-sgr}.

\end{enumerate}
\end{oss}

\end{document}